\documentclass[12pt,reqno]{amsart}
\usepackage{amsmath}
\usepackage{amsfonts}
\usepackage{amssymb}
\usepackage{amsthm}
\usepackage{amscd}
\usepackage{amsbsy}
\usepackage[usenames,dvipsnames,svgnames,x11names,hyperref]{xcolor}
\usepackage{geometry}
\usepackage{graphicx}
\usepackage{hyperref}
\hypersetup{backref=true,
pagebackref=true,
hyperindex=true,
colorlinks=true,
breaklinks=true,
urlcolor=NavyBlue,
linkcolor=Fuchsia,
bookmarks=true,
bookmarksopen=false,
filecolor=black,
citecolor=ForestGreen,
linkbordercolor=red
}
\theoremstyle{plain}
\newtheorem{theorem}{Theorem}[section]
\newtheorem{corollary}[theorem]{Corollary}
\newtheorem{lemma}[theorem]{Lemma}
\newtheorem{proposition}[theorem]{Proposition}
\theoremstyle{definition}

\numberwithin{equation}{section}
\allowdisplaybreaks
\AtBeginDocument{\hypersetup{pdfborder={0 0 0.01}}}
\begin{document}
\title[Sharp stability hierarchies of the $L^2$-Poincar\'e inequalities]{Sharp $L^2$-stability and gradient stability hierarchies of the $L^2$-Poincar\'e inequalities on Euclidean balls and  Gaussian Poincar\'e inequality}
\author{Nguyen Lam}
\address{Nguyen Lam: School of Science and the Environment, Grenfell Campus, Memorial
University of Newfoundland, Corner Brook, NL A2H5G4, Canada}
\email{nlam@mun.ca}
\author{Guozhen Lu}
\address{Guozhen Lu: Department of Mathematics, University of Connecticut, Storrs, CT
06269, USA}
\email{guozhen.lu@uconn.edu}
\date{\today}
\begin{abstract}
We study the sharp $L^2$-Poincar\'e inequalities on Euclidean balls and with respect to the Gaussian measure, and we focus on their stability. In both settings, the sharp constant is the first nonzero eigenvalue of a self-adjoint operator and the set of optimizers is a finite dimensional linear space. Using the spectral decomposition, we establish sharp $L^2$-stability estimates and sharp gradient stability estimates, together with the stability of the stability inequalities in both norms, with explicit optimal constants and with the complete characterization of the equality cases. We then show that this process can be continued. The Poincar\'e deficit is equal to an infinite sum of $L^2$-distances to the increasing spaces of optimizers of the successive stability inequalities. It is also equal to an infinite sum of $L^2$-gradient distances to the same spaces. The coefficients of the first sum are the spectral gaps $\mu_j-\mu_{j-1}$, and the coefficients of the second sum are $\mu_1$ times the gaps of the reciprocal spectrum, $\mu_1\left(\mu_{j-1}^{-1}-\mu_j^{-1}\right)$. On the Euclidean ball, all the constants are expressed through the roots of Bessel functions, and we prove that the first two spectral levels are given by the degree one and the degree two Neumann modes in all dimensions, while the third level is given by the first radial mode when $N=2,3$ and by the degree three modes when $N\geq4$. In the Gaussian case, all the constants are rational numbers and the $L^2$-sum has all coefficients equal to $1$; we prove that this property characterizes the arithmetic spectrum.
\end{abstract}
\maketitle
\section{Introduction}

In \cite{BL85}, Brezis and Lieb raised the question whether the deficit of the sharp Sobolev inequality, that is, the difference between the energy and the optimal constant times the square of the critical $L^{2^*}$-norm, can be bounded from below by the distance from $u$ to the manifold of optimizers. This question was solved affirmatively by Bianchi and Egnell \cite{BE91}, who established the stability estimate
\begin{equation}\label{StabilityS}
\int_{\mathbb{R}^{N}} |\nabla u|^{2}\,dx
- S_{N}\!\left(
\int_{\mathbb{R}^{N}} |u|^{\frac{2N}{N-2}}\,dx
\right)^{\!\frac{N-2}{N}}
\ge c_{BE}
\inf_{U \in E_{Sob}}
\int_{\mathbb{R}^{N}} |\nabla(u - U)|^{2}\,dx,
\end{equation}
where $S_N$ is the optimal Sobolev constant, $E_{Sob}$ is the set of extremal functions and $c_{BE}>0$ is a universal constant. Since then, the stability of functional and geometric inequalities has become a very active research direction. We refer the interested readers to \cite{BE91, BL85, CFLL24, CLT-JFA, CLT24, CLT-MA, CLT-AIM2025, CLTW, DEFFL, DFLL23, DN25, DoLL26, Fat21, FIL16, FN19, FZ-Duke, LLR25, LLR26, MV21, VHN16} and the references therein; this list is far from being complete. We mention in particular that the determination of the sharp stability constants is in general a difficult problem, and that explicit sharp stability constants are known only in a few situations, see for instance \cite{CFLL24, DEFFL, DoLL26, LLR25}.

Poincar\'e inequalities are also a basic tool in analysis and in partial differential equations, and they have been studied in many settings beyond the Euclidean one. For instance, Poincar\'e and Sobolev--Poincar\'e inequalities for H\"ormander vector fields, including weighted versions and the endpoint case $p=1$, were established in \cite{Jerison-Duke, Lu92, Lu94, FLW95a, FLW95b}; the relationship between Poincar\'e type inequalities and representation formulas in spaces of homogeneous type was studied in \cite{FLW96}; Poincar\'e inequalities on product spaces were obtained in \cite{LW98}. In these works, the main questions are the validity of the inequality, the range of the exponents and the dependence of the constant on the geometry. In such general settings, the sharp constants and the optimizers are usually not known, and hence the stability question cannot be formulated as in \eqref{StabilityS}. An exception is the global Poincar\'e inequality on the Heisenberg group $\mathbb{H}^n$ studied in \cite{DLS07}, where it was proved that the best constant and the extremals are the same as those of the Sobolev inequality on $\mathbb{H}^n$, so that, when $p=2$, they are given explicitly by the result of Jerison and Lee \cite{JL88}.

In this paper, we are interested in the stability of the sharp $L^2$-Poincar\'e inequality in two model settings: the Euclidean ball and the Gaussian measure. Let
\[
B_R:=\{x\in\mathbb{R}^N:\ |x|<R\},\qquad N\geq2,
\]
and let $u_{B_R}:=\frac{1}{|B_R|}\int_{B_R}u\,dx$. The sharp $L^2$-Poincar\'e inequality on $B_R$ reads
\begin{equation}\label{PoinBall}
\int_{B_R}|\nabla u|^2\,dx\geq \mu_1(B_R)\int_{B_R}|u-u_{B_R}|^2\,dx,\qquad u\in W^{1,2}(B_R),
\end{equation}
where $\mu_1(B_R)$ is the first nonzero eigenvalue of the Neumann Laplacian
\begin{equation}\label{Neumann}
\begin{cases}
-\Delta u=\lambda u & \text{in } B_R,\\
\partial_\nu u=0 & \text{on } \partial B_R.
\end{cases}
\end{equation}
It is well known that $\mu_1(B_R)=\alpha_N^2/R^2$, see Szeg\"{o} \cite{Sz54}, Weinberger \cite{Wei56} and \cite[Chapter 7]{Hen06}, where $\alpha_N$ is the first positive root of the equation
\begin{equation}\label{alphaN}
J_{N/2}(\alpha)=\alpha J_{N/2+1}(\alpha).
\end{equation}
Here and throughout the paper, for $\nu\geq0$,
\[
J_\nu(\rho):=\sum_{k\geq0}\frac{(-1)^k}{k!\,\Gamma(\nu+k+1)}\left(\frac{\rho}{2}\right)^{2k+\nu}
\]
denotes the Bessel function of the first kind of order $\nu$, that is, the solution of the Bessel equation
\[
\rho^2y''(\rho)+\rho y'(\rho)+(\rho^2-\nu^2)y(\rho)=0
\]
which is regular at the origin, and $0<j_{\nu,1}<j_{\nu,2}<\cdots$ denote its positive zeros. We will use the recurrence relations
\begin{equation}\label{recurrence}
\rho J_\nu'(\rho)=\nu J_\nu(\rho)-\rho J_{\nu+1}(\rho)=\rho J_{\nu-1}(\rho)-\nu J_\nu(\rho),
\end{equation}
which give in particular $\rho J_{\nu+1}(\rho)=2\nu J_\nu(\rho)-\rho J_{\nu-1}(\rho)$. We refer to Watson \cite{Wat44} for these relations and for the other classical facts on Bessel functions that we use below.

We now return to \eqref{PoinBall}. It is also classical, see \cite{Sz54, Wei56, Hen06}, that the equality holds if and only if
\[
u(x)=c+|x|^{1-N/2}J_{N/2}\left(\frac{\alpha_N|x|}{R}\right)\frac{x\cdot \mathbf{a}}{|x|},\qquad c\in\mathbb{R},\ \mathbf{a}\in\mathbb{R}^N.
\]
In particular, for every nonconstant optimizer $u$, the function $u-u_{B_R}$ is antisymmetric with respect to the hyperplane $\{x\cdot\mathbf{a}=0\}$; moreover $u$ is axially symmetric about the axis $\mathbb{R}\mathbf{a}$ and strictly monotone in the polar angle, so that it attains its maximum and its minimum at two antipodal points of $\partial B_R$. See Gir\~ao and Weth \cite{GW06}, where these properties are shown to persist for the Poincar\'e--Sobolev family. We will discuss this family at the end of Section \ref{sec:ball}.

The Gaussian counterpart of \eqref{PoinBall} is the classical Gaussian Poincar\'e inequality
\begin{equation}\label{PoinGauss}
\int_{\mathbb{R}^n}|\nabla u|^2\,d\gamma\geq \inf_{c\in\mathbb{R}}\int_{\mathbb{R}^n}|u-c|^2\,d\gamma,\qquad d\gamma=(2\pi)^{-n/2}e^{-\frac{|x|^2}{2}}dx,
\end{equation}
whose sharp constant is $1$ and whose optimizers are exactly the affine functions $c+\mathbf{a}\cdot x$. The inequality \eqref{PoinGauss} has been established by several classical methods, such as the Hermite polynomial expansion, the semigroup interpolation, and the linearization of the logarithmic Sobolev inequality; see \cite{BGL14, GRO75}.

The stability of the Gaussian Poincar\'e inequality \eqref{PoinGauss} was studied recently in \cite{LLR25}, where the authors proved that
\begin{equation}\label{LLRstab}
\begin{split}
\int_{\mathbb{R}^n}|\nabla u|^2\,d\gamma-\inf_{c\in\mathbb{R}}\int_{\mathbb{R}^n}|u-c|^2\,d\gamma&\geq \frac12\inf_{\mathbf{d}\in\mathbb{R}^n}\int_{\mathbb{R}^n}|\nabla u-\mathbf{d}|^2\,d\gamma\\
&\geq \inf_{c\in\mathbb{R},\ \mathbf{a}\in\mathbb{R}^n}\int_{\mathbb{R}^n}|u-c-\mathbf{a}\cdot x|^2\,d\gamma,
\end{split}
\end{equation}
with sharp constants, and where the equality is attained by nonlinear functions. See also \cite{LLR26} for the weighted Gaussian measures, and \cite{CFLL24, DoLL26} where the Gaussian type Poincar\'e inequalities are the main tool to obtain sharp stability of the Heisenberg uncertainty principle and of the Caffarelli--Kohn--Nirenberg inequalities.

The main purpose of this paper is to show that, in the two settings \eqref{PoinBall} and \eqref{PoinGauss}, the stability problem can be solved {\bf completely}, and that we obtain a hierarchy of sharp inequalities. This is due to the following facts. In both cases, the deficit of the Poincar\'e inequality is a quadratic form, the set of optimizers is a linear space spanned by eigenfunctions, and therefore the deficit can be decomposed exactly along the spectrum. Every truncation of this decomposition is a sharp inequality: the first one is the stability estimate, the second one is the stability of the stability estimate, and so on. Moreover, when we sum up all the terms we recover exactly the deficit. We use two norms to measure the distance to the optimizers, the $L^2$-norm and the $L^2$-gradient norm. They lead to two different infinite sums, with two different families of coefficients.

We now state our main results. We first consider the Euclidean ball. Let
\[
0=\mu_0<\mu_1<\mu_2<\mu_3<\cdots
\]
be the distinct eigenvalues of \eqref{Neumann} and let $E_k$ be the eigenspace of $\mu_k$. We denote
\[
\mathcal{M}_m:=E_0\oplus E_1\oplus\cdots\oplus E_m,\qquad m\geq1,
\]
so that $\mathcal{M}_1$ is exactly the set of optimizers of \eqref{PoinBall}. It is convenient to introduce the deficit
\begin{equation}\label{deficitBall}
\delta_{B_R}(u):=\int_{B_R}|\nabla u|^2\,dx-\mu_1\int_{B_R}|u-u_{B_R}|^2\,dx\geq0.
\end{equation}
Our first result shows that the deficit \eqref{deficitBall} controls both the $L^2$-distance and the $L^2$-gradient distance to the set of optimizers, with explicit sharp constants given by Bessel roots.

\begin{theorem}\label{TstabBall}
Let $N\geq2$ and let $\beta_N$ be the first positive root of the equation
\begin{equation}\label{betaN}
2J_{N/2+1}(\beta)=\beta J_{N/2+2}(\beta).
\end{equation}
Then $\mu_2=\beta_N^2/R^2$, and for all $u\in W^{1,2}(B_R)$ we have
\begin{equation}\label{L2stabBall}
\int_{B_R}|\nabla u|^2\,dx-\frac{\alpha_N^2}{R^2}\int_{B_R}|u-u_{B_R}|^2\,dx\geq \frac{\beta_N^2-\alpha_N^2}{R^2}\inf_{c\in\mathbb{R},\ \phi\in E_1}\int_{B_R}|u-c-\phi|^2\,dx
\end{equation}
and
\begin{equation}\label{GradstabBall}
\int_{B_R}|\nabla u|^2\,dx-\frac{\alpha_N^2}{R^2}\int_{B_R}|u-u_{B_R}|^2\,dx\geq \left(1-\frac{\alpha_N^2}{\beta_N^2}\right)\inf_{\phi\in E_1}\int_{B_R}|\nabla (u-\phi)|^2\,dx.
\end{equation}
The constants $\frac{\beta_N^2-\alpha_N^2}{R^2}$ and $1-\frac{\alpha_N^2}{\beta_N^2}$ are sharp and can be attained. Moreover, the equality holds in \eqref{L2stabBall}, or in \eqref{GradstabBall}, if and only if $u\in\mathcal{M}_2$, that is,
\[
u(x)=c+|x|^{1-N/2}J_{N/2}\left(\frac{\alpha_N|x|}{R}\right)\frac{x\cdot \mathbf{a}}{|x|}+|x|^{1-N/2}J_{N/2+1}\left(\frac{\beta_N|x|}{R}\right)Y_2\left(\frac{x}{|x|}\right),
\]
where $c\in\mathbb{R}$, $\mathbf{a}\in\mathbb{R}^N$ and $Y_2$ is a spherical harmonic of degree two.
\end{theorem}

We note that the $L^2$-stability constant in \eqref{L2stabBall} has the dimension $R^{-2}$, while the gradient stability constant in \eqref{GradstabBall} does not depend on $R$. We also note that $1-\alpha_N^2/\beta_N^2$ tends to $\frac12$ as $N\to\infty$; this follows from the two-sided bounds of Proposition \ref{Pbounds}, which give $N-1\leq\alpha_N^2\leq N+2$ and $2N\leq\beta_N^2\leq2N+8$. For instance, this constant equals $0.63659\ldots$ when $N=2$ and $0.61207\ldots$ when $N=3$, see Corollary \ref{Cor23}.

Our second result is the stability of the stability inequalities \eqref{L2stabBall} and \eqref{GradstabBall}. Here the third eigenvalue $\mu_3$ appears, and its nature is different in low dimensions and in high dimensions.

\begin{theorem}\label{TstabstabBall}
Let $N\geq2$. Let $\rho_{3,1}$ be the first positive root of the equation $3J_{N/2+2}(\rho)=\rho J_{N/2+3}(\rho)$ and let $j_{N/2,1}$ be the first positive zero of $J_{N/2}$. Then
\[
\mu_3=\frac{1}{R^2}\min\left\{j_{N/2,1}^2,\ \rho_{3,1}^2\right\},
\]
and the minimum is $\rho_{3,1}^2$ when $N\geq4$, and $j_{N/2,1}^2$ when $N=2,3$. Moreover, for all $u\in W^{1,2}(B_R)$ we have
\begin{equation}\label{L2L2Ball}
\delta_{B_R}(u)-\frac{\beta_N^2-\alpha_N^2}{R^2}\inf_{c\in\mathbb{R},\ \phi\in E_1}\int_{B_R}|u-c-\phi|^2\,dx\geq \left(\mu_3-\frac{\beta_N^2}{R^2}\right)\inf_{\Phi\in\mathcal{M}_2}\int_{B_R}|u-\Phi|^2\,dx
\end{equation}
and
\begin{equation}\label{GradGradBall}
\begin{split}
&\delta_{B_R}(u)-\left(1-\frac{\alpha_N^2}{\beta_N^2}\right)\inf_{\phi\in E_1}\int_{B_R}|\nabla (u-\phi)|^2\,dx\\
&\qquad\geq \alpha_N^2\left(\frac{1}{\beta_N^2}-\frac{1}{R^2\mu_3}\right)\inf_{\Phi\in\mathcal{M}_2}\int_{B_R}|\nabla (u-\Phi)|^2\,dx.
\end{split}
\end{equation}
The constants are sharp and can be attained, and the equality holds in \eqref{L2L2Ball}, or in \eqref{GradGradBall}, if and only if $u\in\mathcal{M}_3$.
\end{theorem}

We will see in Section \ref{sec:ball} that when $N\geq4$ the new directions in $\mathcal{M}_3$ are the first degree three Neumann modes, while when $N=2,3$ they are given by the first nonconstant radial Neumann eigenfunction $|x|^{1-N/2}J_{N/2-1}(j_{N/2,1}|x|/R)$. We note that in dimension $N=3$ the two values are very close to each other: $j_{3/2,1}^2=20.19\ldots$ and $\rho_{3,1}^2=20.37\ldots$.

Our third result shows that the above process can be continued indefinitely, and that the deficit is recovered exactly.

\begin{theorem}\label{TseriesBall}
Let $N\geq2$. For all $u\in W^{1,2}(B_R)$, we have
\begin{equation}\label{L2seriesBall}
\int_{B_R}|\nabla u|^2\,dx-\mu_1\int_{B_R}|u-u_{B_R}|^2\,dx=\sum_{j=2}^{\infty}\left(\mu_j-\mu_{j-1}\right)\inf_{\Phi\in\mathcal{M}_{j-1}}\int_{B_R}|u-\Phi|^2\,dx
\end{equation}
and
\begin{equation}\label{GradseriesBall}
\int_{B_R}|\nabla u|^2dx-\mu_1\int_{B_R}|u-u_{B_R}|^2dx=\mu_1\sum_{j=2}^{\infty}\left(\frac{1}{\mu_{j-1}}-\frac{1}{\mu_j}\right)\inf_{\Phi\in\mathcal{M}_{j-1}}\int_{B_R}|\nabla (u-\Phi)|^2dx,
\end{equation}
where both series are convergent. Moreover, for every $m\geq2$, the partial sums satisfy the sharp inequalities
\[
\delta_{B_R}(u)\geq\sum_{j=2}^{m}\left(\mu_j-\mu_{j-1}\right)\inf_{\Phi\in\mathcal{M}_{j-1}}\int_{B_R}|u-\Phi|^2\,dx
\]
and
\[
\delta_{B_R}(u)\geq\mu_1\sum_{j=2}^{m}\left(\frac{1}{\mu_{j-1}}-\frac{1}{\mu_j}\right)\inf_{\Phi\in\mathcal{M}_{j-1}}\int_{B_R}|\nabla (u-\Phi)|^2\,dx,
\]
in which the last coefficient cannot be improved and the equality holds if and only if $u\in\mathcal{M}_m$.
\end{theorem}

We now consider the Gaussian setting. Here the Neumann Laplacian is replaced by the Ornstein--Uhlenbeck operator $L=\Delta-x\cdot\nabla$; the eigenvalues of $-L$ are $0,1,2,\ldots$ and its eigenfunctions are the Hermite polynomials. We denote by $\mathcal{P}_m$ the space of polynomials of degree at most $m$ on $\mathbb{R}^n$, so that $\mathcal{P}_1$ is the set of optimizers of \eqref{PoinGauss}, and we denote by $\Pi_m u$ the orthogonal projection of $u$ onto $\mathcal{P}_m$ in $L^2(\gamma)$. Set
\[
\delta_\gamma(u):=\int_{\mathbb{R}^n}|\nabla u|^2\,d\gamma-\inf_{c\in\mathbb{R}}\int_{\mathbb{R}^n}|u-c|^2\,d\gamma\geq0.
\]

\begin{theorem}\label{TGauss}
Let $n\geq1$. For all $u\in W^{1,2}(\mathbb{R}^n,\gamma)$, we have the sharp stability estimates
\begin{equation}\label{L2Gauss}
\delta_\gamma(u)\geq \inf_{c\in\mathbb{R},\ \mathbf{a}\in\mathbb{R}^n}\int_{\mathbb{R}^n}|u-c-\mathbf{a}\cdot x|^2\,d\gamma,\qquad
\delta_\gamma(u)\geq \frac12\inf_{\mathbf{d}\in\mathbb{R}^n}\int_{\mathbb{R}^n}|\nabla u-\mathbf{d}|^2\,d\gamma,
\end{equation}
with equality if and only if $u\in\mathcal{P}_2$; the sharp stability of the stability estimates
\begin{align}
\delta_\gamma(u)-\inf_{c\in\mathbb{R},\ \mathbf{a}\in\mathbb{R}^n}\int_{\mathbb{R}^n}|u-c-\mathbf{a}\cdot x|^2\,d\gamma&\geq \inf_{P\in\mathcal{P}_2}\int_{\mathbb{R}^n}|u-P|^2\,d\gamma,\label{L2L2Gauss}\\
\delta_\gamma(u)-\frac12\inf_{\mathbf{d}\in\mathbb{R}^n}\int_{\mathbb{R}^n}|\nabla u-\mathbf{d}|^2\,d\gamma&\geq \frac16\inf_{P\in\mathcal{P}_2}\int_{\mathbb{R}^n}|\nabla (u-P)|^2\,d\gamma,\label{GradGradGauss}
\end{align}
with equality if and only if $u\in\mathcal{P}_3$; and the two exact identities
\begin{equation}\label{seriesGauss}
\begin{split}
\int_{\mathbb{R}^n}|\nabla u|^2\,d\gamma-\inf_{c\in\mathbb{R}}\int_{\mathbb{R}^n}|u-c|^2\,d\gamma
&=\sum_{j=1}^{\infty}\inf_{P\in\mathcal{P}_j}\int_{\mathbb{R}^n}|u-P|^2\,d\gamma\\
&=\sum_{j=1}^{\infty}\frac{1}{j(j+1)}\inf_{P\in\mathcal{P}_j}\int_{\mathbb{R}^n}|\nabla (u-P)|^2\,d\gamma.
\end{split}
\end{equation}
All the constants above are optimal.
\end{theorem}

The estimates \eqref{L2Gauss} are exactly the stability results in \cite{LLR25}; we recover them here as the first level of the hierarchy. The remaining estimates in Theorem \ref{TGauss} appear to be new. In particular, by Theorem \ref{Tseries} below, every truncation of the two identities in \eqref{seriesGauss} is a sharp inequality, so that the Gaussian Poincar\'e inequality has an infinite family of sharp higher order stability estimates, whose optimizers are the polynomials of increasing degrees. We remark that in the first identity in \eqref{seriesGauss} all the coefficients are equal to $1$. This is a special property of the Gaussian measure. Indeed, we prove in Proposition \ref{Parith} that these coefficients are all equal to the same constant if and only if the positive spectrum is an arithmetic progression. On the Euclidean ball this is not true when $N=2$ or $N=3$, see Corollary \ref{Cor23}.

The proofs of the above theorems are based on a general principle, which is of independent interest. We now describe it. Suppose that $A$ is a nonnegative self-adjoint operator on $L^2(\mu)$ with discrete spectrum $0=\mu_0<\mu_1<\mu_2<\cdots$, with finite dimensional eigenspaces $E_k$, with $E_0$ the constants, and with the associated energy $\mathcal{E}(u)=\langle Au,u\rangle$. Then the deficit $\mathcal{E}(u)-\mu_1\operatorname{Var}_\mu(u)$ is equal to $\sum_{k\geq2}(\mu_k-\mu_1)\|u_k\|_2^2$, where $u_k$ is the component of $u$ in $E_k$. All the stability inequalities of this paper, in both settings, are rearrangements of this identity. The identity also shows that the $L^2$-stability and the gradient stability are equivalent, with an explicit relation between the constants, a fact observed in a slightly different form in \cite{DLLN26}; see Theorem \ref{Tequiv} below. We present this general principle in Section \ref{sec:general}, where Theorems \ref{TstabBall}--\ref{TGauss} are proved, up to the identification of the spectrum in each of the two settings. In Section \ref{sec:ball} we identify the first three Neumann levels of the ball. This is not immediate. Indeed, the Neumann eigenvalues of the ball are given by the roots of infinitely many transcendental equations, one for each degree of spherical harmonics, and the ordering of these roots is not obvious. In Section \ref{sec:gauss} we treat the Gaussian case and prove Proposition \ref{Parith}.

\section{The general principle}\label{sec:general}

In this section, $\mu$ is a finite positive measure on a set $X$, so that the constants belong to $L^2(\mu)$, and $A\geq0$ is a self-adjoint operator on $L^2(\mu)$ with compact resolvent. We denote by
\[
0=\mu_0<\mu_1<\mu_2<\mu_3<\cdots,\qquad \mu_k\to\infty,
\]
the distinct eigenvalues of $A$, by $E_k$ the eigenspace of $\mu_k$, which is of finite dimension, and we assume that $E_0$ is the space of constant functions. Let $\mathcal{E}(u)=\langle Au,u\rangle$ be the associated quadratic form, defined on the form domain $\mathcal{D}$, and let $\bar u$ be the projection of $u$ onto $E_0$, so that $\operatorname{Var}_\mu(u)=\|u-\bar u\|_{L^2(\mu)}^2$. For $u\in\mathcal{D}$ we write
\begin{equation}\label{decomp}
u-\bar u=\sum_{k\geq1}u_k,\qquad u_k\in E_k.
\end{equation}
Since the eigenspaces are orthogonal to each other in $L^2(\mu)$ and also with respect to the form $\mathcal{E}$, we have
\begin{equation}\label{parseval}
\operatorname{Var}_\mu(u)=\sum_{k\geq1}\|u_k\|_2^2,\qquad \mathcal{E}(u)=\sum_{k\geq1}\mu_k\|u_k\|_2^2.
\end{equation}
Hence the sharp Poincar\'e inequality $\mathcal{E}(u)\geq\mu_1\operatorname{Var}_\mu(u)$ holds, with equality if and only if $u-\bar u\in E_1$, and the deficit
\begin{equation}\label{deficit}
\delta(u):=\mathcal{E}(u)-\mu_1\operatorname{Var}_\mu(u)=\sum_{k\geq2}(\mu_k-\mu_1)\|u_k\|_2^2
\end{equation}
does not contain the $E_1$ component. Let $\mathcal{M}_m:=E_0\oplus E_1\oplus\cdots\oplus E_m$. We define the two distances
\[
d_2(u,\mathcal{M}_m)^2:=\inf_{\Phi\in\mathcal{M}_m}\|u-\Phi\|_2^2,\qquad d_{\mathcal{E}}(u,\mathcal{M}_m)^2:=\inf_{\Phi\in\mathcal{M}_m}\mathcal{E}(u-\Phi).
\]
Since $\mathcal{M}_m$ is a spectral subspace, both infima are attained at $\Phi=\bar u+u_1+\cdots+u_m$, and
\begin{equation}\label{distspectral}
d_2(u,\mathcal{M}_m)^2=\sum_{k\geq m+1}\|u_k\|_2^2,\qquad d_{\mathcal{E}}(u,\mathcal{M}_m)^2=\sum_{k\geq m+1}\mu_k\|u_k\|_2^2.
\end{equation}
In particular, from \eqref{deficit} and \eqref{distspectral},
\begin{equation}\label{link}
d_{\mathcal{E}}(u,\mathcal{M}_1)^2=\delta(u)+\mu_1 d_2(u,\mathcal{M}_1)^2.
\end{equation}

We first state the equivalence of the $L^2$-stability and the gradient stability. This result was proved, in a slightly different form, in \cite{DLLN26}. We include the proof here, because by \eqref{link} it is very short.

\begin{theorem}\label{Tequiv}
Let $\alpha>0$. The following statements are equivalent:
\begin{enumerate}
\item[(1)] For all $u\in\mathcal{D}$, $\delta(u)\geq\alpha\, d_2(u,\mathcal{M}_1)^2$.
\item[(2)] For all $u\in\mathcal{D}$, $\delta(u)\geq\dfrac{\alpha}{\mu_1+\alpha}\, d_{\mathcal{E}}(u,\mathcal{M}_1)^2$.
\end{enumerate}
\end{theorem}

\begin{proof}
Assume (1). Then by \eqref{link},
\[
d_{\mathcal{E}}(u,\mathcal{M}_1)^2=\delta(u)+\mu_1d_2(u,\mathcal{M}_1)^2\leq \delta(u)+\frac{\mu_1}{\alpha}\delta(u)=\frac{\mu_1+\alpha}{\alpha}\delta(u),
\]
which is (2). Conversely, assume (2). Then again by \eqref{link},
\[
\delta(u)\geq\frac{\alpha}{\mu_1+\alpha}\left(\delta(u)+\mu_1d_2(u,\mathcal{M}_1)^2\right),
\]
that is, $\frac{\mu_1}{\mu_1+\alpha}\delta(u)\geq \frac{\alpha\mu_1}{\mu_1+\alpha}d_2(u,\mathcal{M}_1)^2$, which is (1).
\end{proof}

\begin{theorem}[Sharp $L^2$-stability and gradient stability]\label{Tstab}
For all $u\in\mathcal{D}$, we have
\begin{equation}\label{L2stab}
\delta(u)\geq(\mu_2-\mu_1)\,d_2(u,\mathcal{M}_1)^2
\end{equation}
and
\begin{equation}\label{Gradstab}
\delta(u)\geq\left(1-\frac{\mu_1}{\mu_2}\right)d_{\mathcal{E}}(u,\mathcal{M}_1)^2.
\end{equation}
The constants $\mu_2-\mu_1$ and $1-\mu_1/\mu_2$ are sharp and can be attained, and the equality holds in \eqref{L2stab} or in \eqref{Gradstab} if and only if $u\in\mathcal{M}_2$. More precisely,
\begin{align}
\delta(u)-(\mu_2-\mu_1)\,d_2(u,\mathcal{M}_1)^2&=\sum_{k\geq3}(\mu_k-\mu_2)\|u_k\|_2^2,\label{L2rem}\\
\delta(u)-\left(1-\frac{\mu_1}{\mu_2}\right)d_{\mathcal{E}}(u,\mathcal{M}_1)^2&=\frac{\mu_1}{\mu_2}\sum_{k\geq3}(\mu_k-\mu_2)\|u_k\|_2^2.\label{Gradrem}
\end{align}
\end{theorem}

\begin{proof}
By \eqref{deficit} and \eqref{distspectral},
\begin{align*}
\delta(u)-(\mu_2-\mu_1)\,d_2(u,\mathcal{M}_1)^2&=\sum_{k\geq2}(\mu_k-\mu_1)\|u_k\|_2^2-(\mu_2-\mu_1)\sum_{k\geq2}\|u_k\|_2^2\\
&=\sum_{k\geq3}(\mu_k-\mu_2)\|u_k\|_2^2\geq0,
\end{align*}
which proves \eqref{L2rem} and \eqref{L2stab}. Similarly,
\begin{align*}
&\delta(u)-\left(1-\frac{\mu_1}{\mu_2}\right)d_{\mathcal{E}}(u,\mathcal{M}_1)^2\\
&=\sum_{k\geq2}\left[(\mu_k-\mu_1)-\left(1-\frac{\mu_1}{\mu_2}\right)\mu_k\right]\|u_k\|_2^2\\
&=\sum_{k\geq2}\mu_1\left(\frac{\mu_k}{\mu_2}-1\right)\|u_k\|_2^2=\frac{\mu_1}{\mu_2}\sum_{k\geq3}(\mu_k-\mu_2)\|u_k\|_2^2\geq0,
\end{align*}
which proves \eqref{Gradrem} and \eqref{Gradstab}. In both cases the right hand side vanishes if and only if $u_k=0$ for all $k\geq3$, that is, $u\in\mathcal{M}_2$. To see that the constants are sharp, we take $u=\Phi+\varepsilon\psi$ with $\Phi\in\mathcal{M}_1$ and $\psi\in E_2\setminus\{0\}$. Then $\delta(u)=(\mu_2-\mu_1)\varepsilon^2\|\psi\|_2^2$, $d_2(u,\mathcal{M}_1)^2=\varepsilon^2\|\psi\|_2^2$ and $d_{\mathcal{E}}(u,\mathcal{M}_1)^2=\mu_2\varepsilon^2\|\psi\|_2^2$, so that the equality holds in both \eqref{L2stab} and \eqref{Gradstab}.
\end{proof}

We note that \eqref{Gradstab} also follows from \eqref{L2stab} and Theorem \ref{Tequiv} with $\alpha=\mu_2-\mu_1$, since $\frac{\alpha}{\mu_1+\alpha}=\frac{\mu_2-\mu_1}{\mu_2}=1-\frac{\mu_1}{\mu_2}$.

The right hand sides of \eqref{L2rem} and \eqref{Gradrem} are again nonnegative quadratic forms. They vanish if and only if $u\in\mathcal{M}_2$, and there is a spectral gap above $\mathcal{M}_2$. Therefore, they can be estimated from below by the distances to $\mathcal{M}_2$. This is the stability of the stability inequalities.

\begin{theorem}[Sharp stability of the stability inequalities]\label{Tstabstab}
For all $u\in\mathcal{D}$, we have
\begin{equation}\label{L2L2}
\delta(u)-(\mu_2-\mu_1)\,d_2(u,\mathcal{M}_1)^2\geq(\mu_3-\mu_2)\,d_2(u,\mathcal{M}_2)^2
\end{equation}
and
\begin{equation}\label{GradGrad}
\delta(u)-\left(1-\frac{\mu_1}{\mu_2}\right)d_{\mathcal{E}}(u,\mathcal{M}_1)^2\geq\mu_1\left(\frac{1}{\mu_2}-\frac{1}{\mu_3}\right)d_{\mathcal{E}}(u,\mathcal{M}_2)^2.
\end{equation}
The constants $\mu_3-\mu_2$ and $\mu_1\left(\frac{1}{\mu_2}-\frac{1}{\mu_3}\right)$ are sharp and can be attained, and the equality holds in \eqref{L2L2} or in \eqref{GradGrad} if and only if $u\in\mathcal{M}_3$. The remainders are respectively $\sum_{k\geq4}(\mu_k-\mu_3)\|u_k\|_2^2$ and $\frac{\mu_1}{\mu_3}\sum_{k\geq4}(\mu_k-\mu_3)\|u_k\|_2^2$.
\end{theorem}

\begin{proof}
By \eqref{L2rem} and \eqref{distspectral},
\begin{align*}
&\delta(u)-(\mu_2-\mu_1)\,d_2(u,\mathcal{M}_1)^2-(\mu_3-\mu_2)\,d_2(u,\mathcal{M}_2)^2\\
&=\sum_{k\geq3}\left[(\mu_k-\mu_2)-(\mu_3-\mu_2)\right]\|u_k\|_2^2=\sum_{k\geq4}(\mu_k-\mu_3)\|u_k\|_2^2\geq0.
\end{align*}
By \eqref{Gradrem} and \eqref{distspectral},
\begin{align*}
&\delta(u)-\left(1-\frac{\mu_1}{\mu_2}\right)d_{\mathcal{E}}(u,\mathcal{M}_1)^2-\mu_1\left(\frac{1}{\mu_2}-\frac{1}{\mu_3}\right)d_{\mathcal{E}}(u,\mathcal{M}_2)^2\\
&=\sum_{k\geq3}\left[\frac{\mu_1}{\mu_2}(\mu_k-\mu_2)-\mu_1\left(\frac{1}{\mu_2}-\frac{1}{\mu_3}\right)\mu_k\right]\|u_k\|_2^2\\
&=\sum_{k\geq3}\mu_1\left(\frac{\mu_k}{\mu_3}-1\right)\|u_k\|_2^2=\frac{\mu_1}{\mu_3}\sum_{k\geq4}(\mu_k-\mu_3)\|u_k\|_2^2\geq0.
\end{align*}
The equality cases and the sharpness of the constants follow as in the proof of Theorem \ref{Tstab}, by testing with $u=\Phi+\varepsilon\psi$, $\Phi\in\mathcal{M}_2$, $\psi\in E_3\setminus\{0\}$.
\end{proof}

We observe that $\mu_1\left(\frac{1}{\mu_2}-\frac{1}{\mu_3}\right)=\frac{\mu_1}{\mu_2}\left(1-\frac{\mu_2}{\mu_3}\right)$, and that the remainder of the gradient inequality after two steps is $\frac{\mu_1}{\mu_3}$ times the remainder of the $L^2$ one. Therefore the process can be repeated. We do this in the next theorem.

\begin{theorem}[The infinite sums]\label{Tseries}
Let $c_j:=\mu_1\left(\frac{1}{\mu_{j-1}}-\frac{1}{\mu_j}\right)$ for $j\geq2$, so that $c_2=1-\frac{\mu_1}{\mu_2}$ and $\sum_{j=2}^{k}c_j=1-\frac{\mu_1}{\mu_k}$. For all $u\in\mathcal{D}$ and all $m\geq2$, we have the exact identities
\begin{align}
\delta(u)&=\sum_{j=2}^{m}(\mu_j-\mu_{j-1})\,d_2(u,\mathcal{M}_{j-1})^2+\sum_{k\geq m+1}(\mu_k-\mu_m)\|u_k\|_2^2,\label{L2finite}\\
\delta(u)&=\sum_{j=2}^{m}c_j\,d_{\mathcal{E}}(u,\mathcal{M}_{j-1})^2+\frac{\mu_1}{\mu_m}\sum_{k\geq m+1}(\mu_k-\mu_m)\|u_k\|_2^2.\label{Gradfinite}
\end{align}
Consequently,
\begin{equation}\label{series}
\delta(u)=\sum_{j=2}^{\infty}(\mu_j-\mu_{j-1})\,d_2(u,\mathcal{M}_{j-1})^2=\mu_1\sum_{j=2}^{\infty}\left(\frac{1}{\mu_{j-1}}-\frac{1}{\mu_j}\right)d_{\mathcal{E}}(u,\mathcal{M}_{j-1})^2,
\end{equation}
where both series are convergent, with nondecreasing partial sums. Moreover, in \eqref{L2finite} and \eqref{Gradfinite} the remainders are nonnegative, they vanish if and only if $u\in\mathcal{M}_m$, and the last coefficients $\mu_m-\mu_{m-1}$ and $c_m$ cannot be replaced by any larger constants.
\end{theorem}

\begin{proof}
Let $k\geq2$ and $m\geq2$. If $k\leq m$, then by telescoping
\[
\mu_k-\mu_1=\sum_{j=2}^{k}(\mu_j-\mu_{j-1}),\qquad \mu_k-\mu_1=\mu_k\left(1-\frac{\mu_1}{\mu_k}\right)=\mu_k\sum_{j=2}^{k}c_j,
\]
while if $k\geq m+1$, then
\[
\mu_k-\mu_1=\sum_{j=2}^{m}(\mu_j-\mu_{j-1})+(\mu_k-\mu_m),\qquad \mu_k-\mu_1=\mu_k\sum_{j=2}^{m}c_j+\frac{\mu_1}{\mu_m}(\mu_k-\mu_m),
\]
where for the last one we used $\mu_k\left(1-\frac{\mu_1}{\mu_m}\right)+\frac{\mu_1}{\mu_m}(\mu_k-\mu_m)=\mu_k-\mu_1$. Inserting these expressions into \eqref{deficit} and exchanging the order of summation, we get
\begin{align*}
\delta(u)&=\sum_{j=2}^{m}(\mu_j-\mu_{j-1})\sum_{k\geq j}\|u_k\|_2^2+\sum_{k\geq m+1}(\mu_k-\mu_m)\|u_k\|_2^2,\\
\delta(u)&=\sum_{j=2}^{m}c_j\sum_{k\geq j}\mu_k\|u_k\|_2^2+\frac{\mu_1}{\mu_m}\sum_{k\geq m+1}(\mu_k-\mu_m)\|u_k\|_2^2,
\end{align*}
which are \eqref{L2finite} and \eqref{Gradfinite} by \eqref{distspectral}. The remainders are clearly nonnegative and vanish if and only if $u_k=0$ for all $k\geq m+1$. Testing with $u=\Phi+\varepsilon\psi$, $\Phi\in\mathcal{M}_{m-1}$, $\psi\in E_m\setminus\{0\}$, we see that the last coefficients are sharp. Finally, both remainders are bounded by
\[
\sum_{k\geq m+1}\mu_k\|u_k\|_2^2\to0\quad\text{as }m\to\infty,
\]
since $\sum_{k}\mu_k\|u_k\|_2^2=\mathcal{E}(u)<\infty$, and therefore the partial sums, which are nondecreasing, converge to $\delta(u)$. This proves \eqref{series}.
\end{proof}

The coefficients of the two series in \eqref{series} are very different. The coefficients $\mu_j-\mu_{j-1}$ of the $L^2$-series are the spectral gaps, which have the same dimension as the eigenvalues and are not summable, since $\mu_j\to\infty$. The coefficients $c_j=\mu_1\left(\frac{1}{\mu_{j-1}}-\frac{1}{\mu_j}\right)$ of the gradient series are dimensionless and positive, they tend to $0$, and they always satisfy
\begin{equation}\label{weightsum}
\sum_{j\geq2}c_j=\lim_{k\to\infty}\left(1-\frac{\mu_1}{\mu_k}\right)=1.
\end{equation}
Hence the gradient series gives the deficit as a weighted average of the gradient distances $d_{\mathcal{E}}(u,\mathcal{M}_{j-1})^2$, with the weights $c_j$.

\section{The Euclidean ball: identification of the spectrum and proofs of Theorems \ref{TstabBall}, \ref{TstabstabBall} and \ref{TseriesBall}}\label{sec:ball}

In this section we apply the results of Section \ref{sec:general} to the Neumann Laplacian on $B_R$, that is, $A=-\Delta$ with the Neumann boundary condition, $\mathcal{E}(u)=\int_{B_R}|\nabla u|^2dx$, $\mathcal{D}=W^{1,2}(B_R)$ and $\bar u=u_{B_R}$. Theorems \ref{TstabBall}, \ref{TstabstabBall} and \ref{TseriesBall} are then exactly Theorems \ref{Tstab}, \ref{Tstabstab} and \ref{Tseries}, once we know that $\mu_1=\alpha_N^2/R^2$, $\mu_2=\beta_N^2/R^2$, and once we know $\mu_3$. The purpose of this section is to establish these facts.

\subsection{The Neumann eigenvalues of the ball}
Let $x=r\omega$ with $r=|x|$ and $\omega\in\mathbb{S}^{N-1}$. Let $Y_\ell$ be a spherical harmonic of degree $\ell\geq0$, that is, the restriction to $\mathbb{S}^{N-1}$ of a homogeneous harmonic polynomial of degree $\ell$. Then $-\Delta_{\mathbb{S}^{N-1}}Y_\ell=\ell(\ell+N-2)Y_\ell$, and the space $\mathcal{H}_\ell$ of all spherical harmonics of degree $\ell$ has dimension
\[
\dim\mathcal{H}_\ell=\binom{N+\ell-1}{\ell}-\binom{N+\ell-3}{\ell-2},
\]
with the convention $\binom{m}{j}=0$ for $j<0$, which is $1$, $N$, $\frac{(N-1)(N+2)}{2}$, $\frac{(N-1)N(N+4)}{6}$ for $\ell=0,1,2,3$. Writing $u(x)=f(r)Y_\ell(\omega)$, the equation $-\Delta u=\lambda u$ becomes
\begin{equation}\label{radialeq}
-f''(r)-\frac{N-1}{r}f'(r)+\frac{\ell(\ell+N-2)}{r^2}f(r)=\lambda f(r),
\end{equation}
and the solution which is regular at the origin is, up to a constant,
\[
f(r)=r^{1-N/2}J_\nu\left(\frac{\rho r}{R}\right),\qquad \rho:=R\sqrt{\lambda},\qquad \nu:=\ell+\frac N2-1,
\]
where $\lambda>0$ is the eigenvalue parameter. The Neumann condition $f'(R)=0$ reads $\rho J_\nu'(\rho)=\left(\frac N2-1\right)J_\nu(\rho)$. Using \eqref{recurrence}, this is equivalent to
\begin{equation}\label{Neumanneq}
\ell\,J_{\nu}(\rho)=\rho\,J_{\nu+1}(\rho),\qquad \nu=\ell+\frac N2-1.
\end{equation}
For $\ell=0$, \eqref{Neumanneq} reads $\rho J_{N/2}(\rho)=0$. For $\ell=1$ and $\ell=2$, \eqref{Neumanneq} is exactly \eqref{alphaN} and \eqref{betaN}, respectively. For $\ell\geq1$ we denote by $\rho_{\ell,1}<\rho_{\ell,2}<\cdots$ the positive roots of \eqref{Neumanneq}; we do not count the trivial root $\rho=0$, since it gives $f\equiv0$. For $\ell=0$ we set $\rho_{0,1}:=0$, which corresponds to the constants, and we let $\rho_{0,2}<\rho_{0,3}<\cdots$ be the positive roots of \eqref{Neumanneq} with $\ell=0$, that is, of $J_{N/2}(\rho)=0$. We set
\[
\lambda_{\ell,k}:=\frac{\rho_{\ell,k}^2}{R^2}.
\]
Then the Neumann spectrum of $B_R$ is the set of all $\lambda_{\ell,k}$, $\ell\geq0$, $k\geq1$, and the eigenspace of $\lambda_{\ell,k}$ is
\[
\left\{r^{1-N/2}J_{\nu}\left(\frac{\rho_{\ell,k}r}{R}\right)Y_\ell(\omega):\ Y_\ell\in\mathcal{H}_\ell\right\}.
\]
The distinct eigenvalues $\mu_1<\mu_2<\cdots$ of Section \ref{sec:general} are the numbers $\lambda_{\ell,k}$ listed in the increasing order and without repetition. We emphasize that \eqref{Neumanneq} is a family of transcendental equations, one for each $\ell$, and the ordering of their roots for different $\ell$ is not obvious. We now study this question.

\subsection{Ordering of the Neumann eigenvalues}
We will use two classical facts on the zeros $j_{\nu,k}$ of the Bessel functions $J_\nu$, see \cite[Chapter XV]{Wat44}. The first one is that the zeros are increasing with respect to the order:
\begin{equation}\label{zerosincrease}
j_{\nu,k}<j_{\nu+1,k},\qquad \nu\geq0,\ k\geq1.
\end{equation}
The second one is the Mittag-Leffler expansion
\begin{equation}\label{ML}
\frac{\rho J_{\nu+1}(\rho)}{J_\nu(\rho)}=\sum_{k\geq1}\frac{2\rho^2}{j_{\nu,k}^2-\rho^2},\qquad \nu\geq0,
\end{equation}
which shows that the function
\[
F_\nu(\rho):=\frac{\rho J_{\nu+1}(\rho)}{J_\nu(\rho)}
\]
is strictly increasing on $(0,j_{\nu,1})$, where it increases from $0$ to $+\infty$, and on each interval $(j_{\nu,k},j_{\nu,k+1})$, where it increases from $-\infty$ to $+\infty$.

It is convenient to rewrite \eqref{Neumanneq} as $F_\nu(\rho)=\ell$, because then the left hand side is the monotone function $F_\nu$ while the right hand side is a constant. The quotient $F_\nu$ is defined only at the points where $J_\nu(\rho)\neq0$, so we first check that every root of \eqref{Neumanneq} satisfies $J_\nu(\rho)\neq0$. Let $\ell\geq1$ and let $\rho>0$ satisfy \eqref{Neumanneq}. If we had $J_\nu(\rho)=0$, then \eqref{Neumanneq} would give $\rho J_{\nu+1}(\rho)=0$, hence also $J_{\nu+1}(\rho)=0$. But $J_\nu$ and $J_{\nu+1}$ have no common positive zero: by \eqref{recurrence}, the two equalities $J_\nu(\rho)=J_{\nu+1}(\rho)=0$ would give $J_\nu(\rho)=J_\nu'(\rho)=0$, and since the Bessel equation
\[
\rho^2y''(\rho)+\rho y'(\rho)+(\rho^2-\nu^2)y(\rho)=0
\]
is regular on $(0,\infty)$, the uniqueness theorem would give $J_\nu\equiv0$, which is not true. Therefore $J_\nu(\rho)\neq0$ at every root of \eqref{Neumanneq}, and we may divide \eqref{Neumanneq} by $J_\nu(\rho)$. Conversely, every solution of $F_\nu(\rho)=\ell$ clearly satisfies \eqref{Neumanneq}. We conclude that, for $\ell\geq1$,
\begin{equation}\label{Feq}
\rho>0\ \text{ satisfies }\eqref{Neumanneq}\qquad\Longleftrightarrow\qquad F_\nu(\rho)=\ell,
\end{equation}
and this is the form of \eqref{Neumanneq} that we use from now on.

\begin{lemma}\label{Lorder}
Let $N\geq2$. The following hold.
\begin{enumerate}
\item[(i)] For every fixed $k\geq1$, $\lambda_{\ell,k}\leq\lambda_{\ell+1,k}$ for all $\ell\geq0$.
\item[(ii)] For every $\ell\geq1$, with $\nu=\ell+\frac N2-1$, we have $\rho_{\ell,1}<j_{\nu,1}$ and $j_{\nu,k}<\rho_{\ell,k+1}<j_{\nu,k+1}$ for $k\geq1$.
\item[(iii)] For every $\ell\geq1$, with $\nu=\ell+\frac N2-1$, we have $\rho_{\ell,1}<j_{\nu-1,1}$.
\item[(iv)] For every $k\geq1$, $\rho_{0,k+1}=j_{N/2,k}$.
\item[(v)] For every $\ell\geq1$, $\rho_{\ell,1}<\rho_{\ell+1,1}$; that is, the numbers $\lambda_{\ell,1}$ are strictly increasing in $\ell\geq1$.
\end{enumerate}
\end{lemma}

\begin{proof}
(i) The number $\lambda_{\ell,k}$ is the $k$-th eigenvalue of the Sturm--Liouville problem \eqref{radialeq} with the Neumann condition at $r=R$, whose Rayleigh quotient is
\[
\frac{\displaystyle\int_0^R\left(|f'(r)|^2+\frac{\ell(\ell+N-2)}{r^2}|f(r)|^2\right)r^{N-1}dr}{\displaystyle\int_0^R|f(r)|^2r^{N-1}dr}.
\]
The potential $\frac{\ell(\ell+N-2)}{r^2}$ is increasing in $\ell$, and the form domain of this quotient does not increase when $\ell$ increases, so (i) follows from the min-max principle.

(ii) By \eqref{Feq}, the roots $\rho_{\ell,k}$ are the solutions of $F_\nu(\rho)=\ell$. By \eqref{ML}, there is exactly one solution in $(0,j_{\nu,1})$ and exactly one solution in each $(j_{\nu,k},j_{\nu,k+1})$.

(iii) By \eqref{zerosincrease}, $j_{\nu-1,1}<j_{\nu,1}$. By the monotonicity of $F_\nu$ on $(0,j_{\nu,1})$, it is enough to show that $F_\nu(j_{\nu-1,1})>\ell$. By \eqref{recurrence}, we have
\[
F_\nu(\rho)=2\nu-\frac{\rho J_{\nu-1}(\rho)}{J_\nu(\rho)},
\]
and at $\rho=j_{\nu-1,1}$ the last term vanishes. Hence $F_\nu(j_{\nu-1,1})=2\nu=2\ell+N-2>\ell$, as desired.

(iv) This is the definition of $\rho_{0,k+1}$, since \eqref{Neumanneq} with $\ell=0$ reads $\rho J_{N/2}(\rho)=0$.

(v) Let $\nu=\ell+\frac N2-1$, so that, by \eqref{Feq}, $F_\nu(\rho_{\ell,1})=\ell$ and $F_{\nu+1}(\rho_{\ell+1,1})=\ell+1$. If $0<\rho<j_{\nu,1}$, then $\rho<j_{\nu,k}<j_{\nu+1,k}$ for every $k$ by \eqref{zerosincrease}, so that all the terms in \eqref{ML} are positive and
\[
0<\frac{2\rho^2}{j_{\nu+1,k}^2-\rho^2}<\frac{2\rho^2}{j_{\nu,k}^2-\rho^2},\qquad k\geq1,
\]
so that $F_{\nu+1}(\rho)<F_\nu(\rho)$ on $(0,j_{\nu,1})$. Suppose now that $\rho_{\ell+1,1}\leq\rho_{\ell,1}$. Since $\rho_{\ell,1}<j_{\nu,1}$ by (ii), and since $F_\nu$ is increasing on $(0,j_{\nu,1})$, we would get
\[
\ell+1=F_{\nu+1}(\rho_{\ell+1,1})<F_\nu(\rho_{\ell+1,1})\leq F_\nu(\rho_{\ell,1})=\ell,
\]
which is impossible. Hence $\rho_{\ell,1}<\rho_{\ell+1,1}$.
\end{proof}

\begin{theorem}\label{Tlevels}
Let $N\geq2$. Then $\mu_1=\lambda_{1,1}=\frac{\alpha_N^2}{R^2}$ and $\mu_2=\lambda_{2,1}=\frac{\beta_N^2}{R^2}$, with
\[
E_1=\left\{|x|^{1-N/2}J_{N/2}\left(\frac{\alpha_N|x|}{R}\right)\frac{x\cdot\mathbf{a}}{|x|}:\ \mathbf{a}\in\mathbb{R}^N\right\}
\]
and
\[
E_2=\left\{|x|^{1-N/2}J_{N/2+1}\left(\frac{\beta_N|x|}{R}\right)Y_2\left(\frac{x}{|x|}\right):\ Y_2\in\mathcal{H}_2\right\}.
\]
In particular $\dim E_1=N$ and $\dim E_2=\frac{(N-1)(N+2)}{2}$.
\end{theorem}

\begin{proof}
By Lemma \ref{Lorder} (iv), $\lambda_{0,2}=j_{N/2,1}^2/R^2$, and by Lemma \ref{Lorder} (i), $\lambda_{0,2}\leq\lambda_{\ell,k}$ for every $\ell\geq0$ and every $k\geq2$. By Lemma \ref{Lorder} (ii) with $\ell=1$ and $\nu=N/2$ we have $\rho_{1,1}<j_{N/2,1}$, and by Lemma \ref{Lorder} (iii) with $\ell=2$ and $\nu=\frac N2+1$ we have $\rho_{2,1}<j_{N/2,1}$. Combining these with Lemma \ref{Lorder} (v), we obtain
\[
\lambda_{0,1}=0<\lambda_{1,1}<\lambda_{2,1}<\lambda_{0,2}\leq\lambda_{\ell,k}\ (k\geq2),\qquad \lambda_{2,1}<\lambda_{\ell,1}\ (\ell\geq3).
\]
Hence $\mu_1=\lambda_{1,1}$ and $\mu_2=\lambda_{2,1}$, and each of these two values is attained by exactly one pair $(\ell,k)$, namely $(1,1)$ and $(2,1)$. The corresponding eigenspaces are therefore the ones displayed above.
\end{proof}

We now consider the third level. Here the answer depends on the dimension.

\begin{theorem}\label{Tthird}
Let $N\geq2$ and let $\rho_{3,1}$, $j_{N/2,1}$ be as in Theorem \ref{TstabstabBall}. Then $\mu_3=\min\{\lambda_{0,2},\lambda_{3,1}\}=\frac{1}{R^2}\min\{j_{N/2,1}^2,\rho_{3,1}^2\}$, and
\begin{equation}\label{criterion}
\rho_{3,1}<j_{N/2,1}\quad\Longleftrightarrow\quad j_{N/2,1}^2<(N+1)(N+2).
\end{equation}
Moreover, \eqref{criterion} holds for every $N\geq4$ and fails for $N=2,3$. Consequently, if $N\geq4$ then
\[
E_3=\left\{|x|^{1-N/2}J_{N/2+2}\left(\frac{\rho_{3,1}|x|}{R}\right)Y_3\left(\frac{x}{|x|}\right):\ Y_3\in\mathcal{H}_3\right\},
\]
which has dimension $\dim E_3=\frac{(N-1)N(N+4)}{6}$,
while if $N=2,3$ then $E_3$ is the one dimensional space spanned by the radial function $|x|^{1-N/2}J_{N/2-1}\left(\frac{j_{N/2,1}|x|}{R}\right)$.
\end{theorem}

\begin{proof}
Once $\lambda_{0,1}$, $\lambda_{1,1}$ and $\lambda_{2,1}$ are removed, Lemma \ref{Lorder} (i) and (v) show that $\mu_3$ is the minimum of $\lambda_{0,2}$ and $\lambda_{3,1}$. Let $\nu=\frac N2+2$ and $\rho_0=j_{N/2,1}=j_{\nu-2,1}$. By \eqref{zerosincrease}, $\rho_0<j_{\nu,1}$, so by Lemma \ref{Lorder} (ii) and the monotonicity of $F_\nu$ on $(0,j_{\nu,1})$, we have $\rho_{3,1}<\rho_0$ if and only if $F_\nu(\rho_0)>3$. Using \eqref{recurrence} twice,
\[
F_\nu(\rho)=2\nu-\frac{\rho J_{\nu-1}(\rho)}{J_\nu(\rho)},\qquad \rho J_\nu(\rho)=2(\nu-1)J_{\nu-1}(\rho)-\rho J_{\nu-2}(\rho).
\]
At $\rho=\rho_0$ we have $J_{\nu-2}(\rho_0)=0$, hence $\rho_0J_\nu(\rho_0)=2(\nu-1)J_{\nu-1}(\rho_0)$, and so
\[
F_\nu(\rho_0)=2\nu-\frac{\rho_0^2}{2(\nu-1)}=N+4-\frac{j_{N/2,1}^2}{N+2}.
\]
This is larger than $3$ if and only if $j_{N/2,1}^2<(N+1)(N+2)$, which proves \eqref{criterion}.

To show that \eqref{criterion} holds for $N\geq4$, we use the elementary estimate
\begin{equation}\label{jbound}
j_{\nu,1}^2<2(\nu+1)(\nu+3),\qquad \nu>0.
\end{equation}
Indeed, $j_{\nu,1}^2$ is the first eigenvalue of the Bessel operator
\[
-g''(t)-\frac1tg'(t)+\frac{\nu^2}{t^2}g(t)=\lambda g(t)\ \text{ on }(0,1),\qquad g(1)=0,
\]
whose eigenfunctions are $g(t)=J_\nu(\sqrt\lambda\,t)$, so that
\[
j_{\nu,1}^2=\min\frac{\displaystyle\int_0^1\left(|g'(t)|^2+\frac{\nu^2}{t^2}|g(t)|^2\right)t\,dt}{\displaystyle\int_0^1|g(t)|^2t\,dt}.
\]
Testing with $g(t)=t^\nu(1-t^2)$, which belongs to the form domain since $\nu>0$, and using $\int_0^1t^{2\nu+2j-1}dt=\frac{1}{2\nu+2j}$, we get
\begin{align*}
\int_0^1\left(|g'|^2+\frac{\nu^2}{t^2}g^2\right)t\,dt&=\frac{2\nu^2}{2\nu}-\frac{2\nu(\nu+2)+2\nu^2}{2\nu+2}+\frac{(\nu+2)^2+\nu^2}{2\nu+4}=\frac{2}{\nu+2},\\
\int_0^1g^2\,t\,dt&=\frac{1}{2\nu+2}-\frac{2}{2\nu+4}+\frac{1}{2\nu+6}=\frac{1}{(\nu+1)(\nu+2)(\nu+3)},
\end{align*}
so that the quotient equals $2(\nu+1)(\nu+3)$. The inequality in \eqref{jbound} is strict since $g$ is not an eigenfunction. Now with $\nu=N/2$, \eqref{jbound} gives
\[
j_{N/2,1}^2<\frac{(N+2)(N+6)}{2}\leq(N+1)(N+2)\quad\text{as soon as }N\geq4.
\]
It remains to check that \eqref{criterion} fails for $N=2$ and for $N=3$. For $N=2$ we use the classical Rayleigh sum
\[
\sum_{k\geq1}\frac{1}{j_{\nu,k}^4}=\frac{1}{16(\nu+1)^2(\nu+2)},\qquad \nu\geq0,
\]
see \cite[Chapter XV]{Wat44}, which gives $j_{\nu,1}^4\geq16(\nu+1)^2(\nu+2)$ and therefore
\[
j_{1,1}^2\geq\sqrt{192}=8\sqrt3=13.85\ldots>12=(N+1)(N+2).
\]
For $N=3$ we have $J_{3/2}(\rho)=\sqrt{\frac{2}{\pi\rho}}\left(\frac{\sin\rho}{\rho}-\cos\rho\right)$, so that $j_{3/2,1}$ is the first positive root of $\tan\rho=\rho$, which lies in $\left(\pi,\frac{3\pi}{2}\right)$. On this interval the function $\rho\mapsto\tan\rho-\rho$ is increasing, since its derivative is $\tan^2\rho\geq0$, and it tends to $+\infty$ at the right endpoint. At the point $\rho=\sqrt{20}\in\left(\pi,\frac{3\pi}{2}\right)$ we have $\tan\sqrt{20}=4.081\ldots<4.472\ldots=\sqrt{20}$, hence $j_{3/2,1}>\sqrt{20}$, that is,
\[
j_{3/2,1}^2>20=(N+1)(N+2).
\]
Finally, we identify $E_3$. By Lemma \ref{Lorder} (ii) we have $\rho_{1,2}>j_{\nu,1}=j_{N/2,1}=\rho_{0,2}$ with $\nu=N/2$, so that, using Lemma \ref{Lorder} (i) and $\lambda_{\ell,k}\geq\lambda_{\ell,2}$ for $k\geq2$,
\[
\lambda_{0,2}<\lambda_{1,2}\leq\lambda_{\ell,2}\leq\lambda_{\ell,k}\qquad\text{for all }\ell\geq1\text{ and }k\geq2.
\]
Together with Lemma \ref{Lorder} (v) this shows that $\mu_3$ is attained by exactly one pair $(\ell,k)$, namely $(3,1)$ when $N\geq4$ and $(0,2)$ when $N=2,3$, which gives the two descriptions of $E_3$. The proof is completed.
\end{proof}

We now explain the reason for this change between $N=3$ and $N=4$. It is a difference of scaling in $N$. In the sector of degree $\ell$, the harmonic polynomial $r^\ell Y_\ell$ is an admissible test function for the Neumann Rayleigh quotient, and it gives $\lambda_{\ell,1}\leq\frac{\ell(2\ell+N)}{R^2}$, which is linear in $N$; see Proposition \ref{Pbounds} below. On the other hand, a radial eigenfunction must oscillate in the variable $r$, and the term $\frac{N-1}{r}f'$ in \eqref{radialeq} makes its eigenvalue larger. Since $j_{\nu,1}>\nu$ \cite[Chapter XV]{Wat44}, we have $\lambda_{0,2}=\frac{j_{N/2,1}^2}{R^2}>\frac{N^2}{4R^2}$, and this is quadratic in $N$. Hence, in large dimensions, the first several distinct Neumann eigenvalues are $\lambda_{1,1}<\lambda_{2,1}<\lambda_{3,1}<\cdots$, and the first radial eigenvalue $\lambda_{0,2}$ appears much later in the list.

\subsection{Two-sided bounds and the low dimensions}
The eigenvalues $\lambda_{\ell,1}$ are not explicit. However, they can be estimated from above and from below by two elementary quantities, and these estimates are enough to describe the size of the sharp constants in all dimensions.

\begin{proposition}\label{Pbounds}
Let $N\geq2$. For every $\ell\geq1$,
\begin{equation}\label{twosided}
\frac{\ell(\ell+N-2)}{R^2}\leq\lambda_{\ell,1}\leq\frac{\ell(2\ell+N)}{R^2}.
\end{equation}
In particular,
\[
\frac{N-1}{R^2}\leq\mu_1=\frac{\alpha_N^2}{R^2}\leq\frac{N+2}{R^2},\qquad
\frac{2N}{R^2}\leq\mu_2=\frac{\beta_N^2}{R^2}\leq\frac{2N+8}{R^2},
\]
so that the sharp constants of Theorem \ref{TstabBall} satisfy
\[
\frac{N-2}{R^2}\leq\mu_2-\mu_1\leq\frac{N+9}{R^2},\qquad
\frac{N-2}{2N}\leq1-\frac{\mu_1}{\mu_2}\leq\frac{N+9}{2N+8}.
\]
Consequently
\[
\lim_{N\to\infty}\left(1-\frac{\mu_1}{\mu_2}\right)=\frac12.
\]
\end{proposition}

\begin{proof}
Both bounds in \eqref{twosided} come from the Rayleigh quotient of the degree $\ell$ sector,
\[
Q_\ell(f)=\frac{\displaystyle\int_0^R\left(|f'(r)|^2+\frac{\ell(\ell+N-2)}{r^2}|f(r)|^2\right)r^{N-1}dr}{\displaystyle\int_0^R|f(r)|^2r^{N-1}dr},
\]
whose minimum is $\lambda_{\ell,1}$. Since $0<r<R$ we have $\frac{1}{r^2}>\frac{1}{R^2}$, hence $Q_\ell(f)\geq\frac{\ell(\ell+N-2)}{R^2}$ for every admissible $f$, which is the lower bound. For the upper bound we test with $f(r)=r^\ell$, that is, with the harmonic polynomial $u=r^\ell Y_\ell$. Since $\Delta u=0$, the Green formula gives
\[
\int_{B_R}|\nabla u|^2dx=\int_{\partial B_R}u\,\partial_\nu u\,d\sigma=\ell R^{2\ell+N-2}\int_{\mathbb{S}^{N-1}}|Y_\ell|^2d\sigma,
\]
while
\[
\int_{B_R}|u|^2dx=\frac{R^{2\ell+N}}{2\ell+N}\int_{\mathbb{S}^{N-1}}|Y_\ell|^2d\sigma,
\]
so that the quotient equals $\frac{\ell(2\ell+N)}{R^2}$. Taking $\ell=1$ and $\ell=2$ in \eqref{twosided}, and using Theorem \ref{Tlevels}, we obtain the bounds for $\mu_1$ and $\mu_2$, hence those for $\mu_2-\mu_1$. For the gradient constant, $\frac{\mu_1}{\mu_2}\leq\frac{N+2}{2N}$ and $\frac{\mu_1}{\mu_2}\geq\frac{N-1}{2N+8}$, so that
\[
1-\frac{N+2}{2N}=\frac{N-2}{2N}\leq1-\frac{\mu_1}{\mu_2}\leq1-\frac{N-1}{2N+8}=\frac{N+9}{2N+8},
\]
and both ends tend to $\frac12$ as $N\to\infty$.
\end{proof}

In dimension three the Bessel functions involved are elementary, and in dimension two they are the classical Bessel functions of integer order. In both cases all the constants can be computed numerically.

\begin{corollary}\label{Cor23}
Let $R=1$.
\begin{enumerate}
\item[(i)] If $N=2$, then \eqref{Neumanneq} reads $J_\ell'(\rho)=0$, and $\alpha_2=1.84118\ldots$, $\beta_2=3.05423\ldots$ are the first positive zeros of $J_1'$ and of $J_2'$, while $j_{1,1}=3.83170\ldots$ is the first positive zero of $J_1$. Hence
\[
\mu_1=3.38995\ldots,\qquad \mu_2=9.32836\ldots,\qquad \mu_3=14.68197\ldots,
\]
and the four sharp constants of Theorems \ref{TstabBall} and \ref{TstabstabBall} are
\[
\mu_2-\mu_1=5.93840\ldots,\qquad 1-\frac{\mu_1}{\mu_2}=0.63659\ldots,
\]
\[
\mu_3-\mu_2=5.35360\ldots,\qquad \mu_1\left(\frac{1}{\mu_2}-\frac{1}{\mu_3}\right)=0.13251\ldots.
\]
\item[(ii)] If $N=3$, then \eqref{Neumanneq} is elementary: it reads $\tan\rho=\rho$ for $\ell=0$, $\tan\rho=\frac{2\rho}{2-\rho^2}$ for $\ell=1$, and $\tan\rho=\frac{\rho(\rho^2-9)}{4\rho^2-9}$ for $\ell=2$, with first positive roots $4.49340\ldots$, $\alpha_3=2.08157\ldots$ and $\beta_3=3.34209\ldots$ respectively. Hence
\[
\mu_1=4.33295\ldots,\qquad \mu_2=11.16959\ldots,\qquad \mu_3=20.19072\ldots,
\]
and the four sharp constants are
\[
\mu_2-\mu_1=6.83663\ldots,\qquad 1-\frac{\mu_1}{\mu_2}=0.61207\ldots,
\]
\[
\mu_3-\mu_2=9.02113\ldots,\qquad \mu_1\left(\frac{1}{\mu_2}-\frac{1}{\mu_3}\right)=0.17332\ldots.
\]
\end{enumerate}
In both cases $\mu_3$ is the first nonconstant radial level, and the next one is the first degree three level: $\mu_4=\lambda_{3,1}=17.64998\ldots$ when $N=2$, and $\mu_4=\lambda_{3,1}=20.37709\ldots$ when $N=3$, which is very close to $\mu_3=20.19072\ldots$.
\end{corollary}

\begin{proof}
For $N=2$ we have $\frac N2-1=0$, so the Neumann condition $\rho J_\nu'(\rho)=\left(\frac N2-1\right)J_\nu(\rho)$ becomes $J_\ell'(\rho)=0$, and $J_0'=-J_1$ gives $\rho_{0,2}=j_{1,1}$. For $N=3$ we have $\nu=\ell+\frac12$, and $r^{-1/2}J_{\ell+1/2}(\rho r)$ is a spherical Bessel function, so that \eqref{Neumanneq} reduces to the three elementary equations displayed above. The numerical values are obtained from these equations, and the constants follow from Theorems \ref{Tlevels} and \ref{Tthird}. For the last statement, we note that $\lambda_{\ell,k}\geq\lambda_{1,2}$ for all $\ell\geq1$ and $k\geq2$ by Lemma \ref{Lorder} (i), that $\lambda_{\ell,1}>\lambda_{3,1}$ for $\ell\geq4$ by Lemma \ref{Lorder} (v), that $\rho_{1,2}<j_{N/2,2}=\rho_{0,3}\leq\rho_{0,k}$ for $k\geq3$ by Lemma \ref{Lorder} (ii) and (iv), and that $\rho_{3,1}<\rho_{1,2}$ numerically: $4.20118\ldots<5.33144\ldots$ when $N=2$ and $4.51409\ldots<5.94036\ldots$ when $N=3$. Hence $\mu_4=\lambda_{3,1}$ in both cases.
\end{proof}

\subsection{Proofs of Theorems \ref{TstabBall}, \ref{TstabstabBall} and \ref{TseriesBall}}
By Theorem \ref{Tlevels}, $\mu_1=\alpha_N^2/R^2$, $\mu_2=\beta_N^2/R^2$, and $\mathcal{M}_2$ is as described in Theorem \ref{TstabBall}. Therefore Theorem \ref{TstabBall} is Theorem \ref{Tstab}, with
\[
\mu_2-\mu_1=\frac{\beta_N^2-\alpha_N^2}{R^2},\qquad 1-\frac{\mu_1}{\mu_2}=1-\frac{\alpha_N^2}{\beta_N^2}.
\]
By Theorem \ref{Tthird}, $\mu_3$ is as described in Theorem \ref{TstabstabBall}, and Theorem \ref{TstabstabBall} is Theorem \ref{Tstabstab}, with
\[
\mu_3-\mu_2=\mu_3-\frac{\beta_N^2}{R^2},\qquad \mu_1\left(\frac{1}{\mu_2}-\frac{1}{\mu_3}\right)=\alpha_N^2\left(\frac{1}{\beta_N^2}-\frac{1}{R^2\mu_3}\right).
\]
Finally, Theorem \ref{TseriesBall} is Theorem \ref{Tseries}. \qed

The results of Section \ref{sec:general} do not use the shape of the domain. Hence Theorems \ref{TstabBall}, \ref{TstabstabBall} and \ref{TseriesBall} also hold on any bounded Lipschitz domain $\Omega\subset\mathbb{R}^N$, with $\mu_k$ the distinct Neumann eigenvalues of $\Omega$ and $E_k$ their eigenspaces. For the ball we have in addition $\dim E_1=N$, and all the levels are given by the roots of Bessel functions.

We also note that the method of Section \ref{sec:general} is specific to the exponent $p=2$. For $p\neq2$, the sharp constant of the $L^p$-Poincar\'e inequality is not an eigenvalue of a linear operator, and the optimizers, when they exist, do not form a linear space in general. For instance, when $p=1$, the sharp $L^1$-Poincar\'e inequality on a rectangle $\prod_{i=1}^N[a_i,b_i]$ was established in \cite[Appendix 1]{LW98}:
\[
\int_{\prod_{i=1}^N[a_i,b_i]}\left|f-\frac{1}{\prod_{i=1}^N(b_i-a_i)}\int_{\prod_{i=1}^N[a_i,b_i]}f\,dy\right|dx\leq\sum_{i=1}^N\frac{b_i-a_i}{2}\int_{\prod_{i=1}^N[a_i,b_i]}\left|\frac{\partial f}{\partial x_i}\right|dx,
\]
where each constant $\frac{b_i-a_i}{2}$ is sharp. This constant is not attained by Lipschitz functions, and the optimizing sequences converge to step functions. Hence the stability of this inequality needs a different approach, and we do not study it here.

We end this section with the Poincar\'e--Sobolev family
\[
\Lambda_q(B_R):=\inf\left\{\frac{\displaystyle\int_{B_R}|\nabla u|^2\,dx}{\left(\displaystyle\int_{B_R}|u|^q\,dx\right)^{2/q}}:\ u\in W^{1,2}(B_R)\setminus\{0\},\ \int_{B_R}u\,dx=0\right\},
\]
where $2\leq q\leq\frac{2N}{N-2}$ if $N\geq3$ and $2\leq q<\infty$ if $N=2$, whose endpoint $q=2$ is exactly \eqref{PoinBall}, that is, $\Lambda_2(B_R)=\mu_1(B_R)$. Here the situation is different, and the method of Section \ref{sec:general} does not apply, because for $q>2$ the deficit is no longer a quadratic form and the set of optimizers is no longer a linear space. By \cite{GW06}, for $q>2$ the optimizers are still axially symmetric and strictly monotone in the polar angle. Moreover, when $q$ is close to $2$, they are antisymmetric with respect to the equatorial hyperplane, and they are unique up to a rotation and a multiplicative constant. However, when $N=2$ and $q$ is large, this antisymmetry does not hold. Thus, for $q>2$, the set of optimizers is a nonlinear manifold, and it depends on $q$. To obtain a stability result of Bianchi--Egnell type \cite{BE91}, one would first need the nondegeneracy of this manifold. We do not study this problem here.

\section{The Gaussian measure: proof of Theorem \ref{TGauss}}\label{sec:gauss}

Let $\gamma$ be the standard Gaussian measure on $\mathbb{R}^n$ and let $L=\Delta-x\cdot\nabla$ be the Ornstein--Uhlenbeck operator. Then $L$ is self-adjoint on $L^2(\gamma)$ and, by integration by parts,
\[
\int_{\mathbb{R}^n}(-Lu)v\,d\gamma=\int_{\mathbb{R}^n}\nabla u\cdot\nabla v\,d\gamma.
\]
We apply Section \ref{sec:general} with $A=-L$, $\mathcal{E}(u)=\int_{\mathbb{R}^n}|\nabla u|^2d\gamma$, $\mathcal{D}=W^{1,2}(\mathbb{R}^n,\gamma)$ and $\bar u=\int u\,d\gamma$. It is well known that the spectrum of $-L$ is $\{0,1,2,\ldots\}$, and that the eigenspace $E_k$ of the eigenvalue $k$ is the span of the products $He_{k_1}(x_1)\cdots He_{k_n}(x_n)$ with $k_1+\cdots+k_n=k$, where $He_j$ is the Hermite polynomial of degree $j$ (with the probabilists' normalization $He_j''-xHe_j'=-jHe_j$), of dimension $\binom{n+k-1}{k}$; see \cite[Section 2.7]{BGL14}. Consequently
\[
\mu_k=k,\qquad \mathcal{M}_m=E_0\oplus\cdots\oplus E_m=\mathcal{P}_m,
\]
and, as in \eqref{decomp}, $u_k$ denotes the component of $u$ in $E_k$, that is, the projection of $u$ onto the $k$-th Wiener chaos. Thus, for every $u\in W^{1,2}(\mathbb{R}^n,\gamma)$,
\begin{equation}\label{chaos}
\inf_{c\in\mathbb{R}}\int_{\mathbb{R}^n}|u-c|^2d\gamma=\sum_{k\geq1}\|u_k\|_2^2,\qquad\int_{\mathbb{R}^n}|\nabla u|^2d\gamma=\sum_{k\geq1}k\|u_k\|_2^2,
\end{equation}
so that the Gaussian deficit is
\begin{equation}\label{gdeficit}
\delta_\gamma(u)=\sum_{k\geq2}(k-1)\|u_k\|_2^2.
\end{equation}
The first projections can be written explicitly. By the Gaussian integration by parts formula $\int_{\mathbb{R}^n}\partial_iu\,d\gamma=\int_{\mathbb{R}^n}x_iu\,d\gamma$, we have
\begin{equation}\label{Pi1}
\Pi_1u=\int_{\mathbb{R}^n}u\,d\gamma+\left(\int_{\mathbb{R}^n}u\,x\,d\gamma\right)\cdot x,\qquad \nabla(u-\Pi_1u)=\nabla u-\int_{\mathbb{R}^n}\nabla u\,d\gamma,
\end{equation}
and, integrating by parts twice, with $m_{ij}:=\int_{\mathbb{R}^n}\partial_{ij}u\,d\gamma=\int_{\mathbb{R}^n}u\,(x_ix_j-\delta_{ij})\,d\gamma$,
\begin{equation}\label{Pi2}
\Pi_2u=\Pi_1u+\frac12\sum_{i,j=1}^n m_{ij}\,(x_ix_j-\delta_{ij}).
\end{equation}
In particular, the infimum over $\mathbf{d}\in\mathbb{R}^n$ in \eqref{L2Gauss} is attained at $\mathbf{d}=\int\nabla u\,d\gamma=\int u\,x\,d\gamma$, so that
\[
\inf_{\mathbf{d}\in\mathbb{R}^n}\int_{\mathbb{R}^n}|\nabla u-\mathbf{d}|^2d\gamma=\int_{\mathbb{R}^n}|\nabla(u-\Pi_1u)|^2d\gamma=d_{\mathcal{E}}(u,\mathcal{P}_1)^2.
\]

\begin{proof}[Proof of Theorem \ref{TGauss}]
Since $\mu_k=k$, we have
\[
\mu_2-\mu_1=1,\qquad 1-\frac{\mu_1}{\mu_2}=\frac12,\qquad \mu_3-\mu_2=1,\qquad \mu_1\left(\frac{1}{\mu_2}-\frac{1}{\mu_3}\right)=\frac16,
\]
and in general $\mu_j-\mu_{j-1}=1$, $c_j=\frac{1}{j-1}-\frac1j=\frac{1}{(j-1)j}$. Then \eqref{L2Gauss} is Theorem \ref{Tstab}, \eqref{L2L2Gauss} and \eqref{GradGradGauss} are Theorem \ref{Tstabstab}, and \eqref{seriesGauss} is Theorem \ref{Tseries} after the change of index $j\mapsto j+1$. For the reader's convenience, we verify the two identities in \eqref{seriesGauss} directly. By \eqref{distspectral}, $\inf_{P\in\mathcal{P}_j}\int|u-P|^2d\gamma=\sum_{k>j}\|u_k\|_2^2$, and summing over $j\geq1$, the $k$-th component is counted exactly $k-1$ times, which gives \eqref{gdeficit}. Similarly, $\inf_{P\in\mathcal{P}_j}\int|\nabla(u-P)|^2d\gamma=\sum_{k>j}k\|u_k\|_2^2$ and
\[
\sum_{j=1}^{k-1}\frac{1}{j(j+1)}=1-\frac1k,
\]
so that
\[
\sum_{j\geq1}\frac{1}{j(j+1)}\sum_{k>j}k\|u_k\|_2^2=\sum_{k\geq2}k\left(1-\frac1k\right)\|u_k\|_2^2=\sum_{k\geq2}(k-1)\|u_k\|_2^2,
\]
as desired.
\end{proof}

We now show that the coefficients in the first identity of \eqref{seriesGauss} are all equal to $1$ exactly when the spectrum is an arithmetic progression.

\begin{proposition}\label{Parith}
In the setting of Section \ref{sec:general}, the identity
\[
\delta(u)=c\sum_{j\geq2}d_2(u,\mathcal{M}_{j-1})^2\qquad\text{for all }u\in\mathcal{D},
\]
holds for some constant $c>0$ if and only if $\mu_k=\mu_1+(k-1)c$ for all $k\geq1$, that is, if and only if the positive spectrum is an arithmetic progression with the common difference $c$. In particular, such an identity does not hold for the Neumann Laplacian on a Euclidean ball when $N=2$ or $N=3$.
\end{proposition}

\begin{proof}
If $\mu_k=\mu_1+(k-1)c$, then $\mu_j-\mu_{j-1}=c$ for all $j\geq2$ and the identity is \eqref{series}. Conversely, let $u\in E_k$ with $k\geq2$. Then $\delta(u)=(\mu_k-\mu_1)\|u\|_2^2$, while $d_2(u,\mathcal{M}_{j-1})^2=\|u\|_2^2$ if $j\leq k$ and $=0$ if $j>k$. Hence the identity gives $\mu_k-\mu_1=c(k-1)$ for every $k\geq2$. On the ball with $N=2$ or $N=3$, Corollary \ref{Cor23} gives $\mu_2-\mu_1\neq\mu_3-\mu_2$, so that the positive spectrum is not an arithmetic progression.
\end{proof}

We conclude with a comparison of the two settings. In both cases, the equality set at the level $m$ is the finite dimensional linear space $\mathcal{M}_m$, and the union of the $\mathcal{M}_m$ is dense. On the ball, $\mathcal{M}_m$ is spanned by Bessel modes and the coefficients $\mu_j-\mu_{j-1}$ of the $L^2$-series are not given by any simple formula; for the Gaussian measure, $\mathcal{M}_m$ is the space of polynomials of degree at most $m$ and all the coefficients of the $L^2$-series are equal to $1$. The coefficients $c_j$ of the gradient series are in both cases the dimensionless numbers $\mu_1\left(\frac{1}{\mu_{j-1}}-\frac{1}{\mu_j}\right)$, which sum up to $1$ by \eqref{weightsum}. For the Gaussian measure they are $\frac12,\frac16,\frac1{12},\ldots$, and on the unit ball in $\mathbb{R}^3$ they are $0.61207\ldots$, $0.17332\ldots$, and so on, by Corollary \ref{Cor23}. Finally, on the ball the $L^2$-constants have the dimension $R^{-2}$, and the gradient constants depend only on $N$. In the Gaussian case there is no scaling parameter, and all the constants are rational numbers.

\noindent\textbf{Acknowledgements.} N. Lam was partially supported by an NSERC Discovery Grant. G. Lu was partially supported by grants from the Simons Foundation.


\begin{thebibliography}{99}

\bibitem{BGL14}
D. Bakry, I. Gentil, M. Ledoux,
Analysis and Geometry of Markov Diffusion Operators,
Grundlehren Math. Wiss., vol.~348, Springer, Cham, 2014.

\bibitem{BE91}
G. Bianchi, H. Egnell,
A note on the Sobolev inequality,
J. Funct. Anal. {\bf 100} (1991), 18--24.

\bibitem{BL85}
H. Brezis, E. H. Lieb,
Sobolev inequalities with remainder terms,
J. Funct. Anal. {\bf 62} (1985), 73--86.

 

\bibitem{CFLL24}
C. Cazacu, J. Flynn, N. Lam, G. Lu,
Caffarelli-Kohn-Nirenberg identities, inequalities and their stabilities,
J. Math. Pures Appl. (9) {\bf 182} (2024), 253--284.

\bibitem{CLT-JFA}
L. Chen, G. Lu, H. Tang, Sharp stability of log-Sobolev and Moser-Onofri inequalities on the sphere, J. Funct. Anal. {\bf 285} (2023), no. 5, Paper No. 110022, 24 pp.

\bibitem{CLT24}
L. Chen, G. Lu, H. Tang,
Stability of Hardy-Littlewood-Sobolev inequalities with explicit lower bounds,
Adv. Math. {\bf 450} (2024), Paper No. 109778, 28 pp.

\bibitem{CLT-MA}
L. Chen, G. Lu, H. Tang,
Optimal stability of Hardy-Littlewood-Sobolev and Sobolev inequalities of arbitrary orders with dimension-dependent constants,
Math. Ann. {\bf 394} (2026), no. 4, Paper No. 77, 46 pp.

\bibitem{CLT-AIM2025}
L. Chen, G. Lu, H. Tang,
Optimal asymptotic lower bound for stability of fractional Sobolev inequality and the global stability of log-Sobolev inequality on the sphere,
Adv. Math. {\bf 479} (2025), part B, Paper No. 110438, 32 pp.

 
\bibitem{CLTW}
L. Chen, G. Lu, H. Tang, B. Wang,
Asymptotically sharp stability of Sobolev inequalities on the Heisenberg group with dimension-dependent constants,
J. Math. Pures Appl. (9) {\bf 206} (2026), Paper No. 103832, 29 pp.

\bibitem{DFLL23}
A. X. Do, J. Flynn, N. Lam, G. Lu,
$L^{p}$-Caffarelli-Kohn-Nirenberg inequalities and their stabilities,
arXiv:2310.07083.

\bibitem{DLLN26}
A. X. Do, N. Lam, G. Lu, V. H. Nguyen,
Caffarelli-Kohn-Nirenberg and weighted Gaussian Poincar\'e inequalities: a complete characterization of sharp $L^2$ stability and $L^p$ extensions, arXiv:2606.08939.

\bibitem{DoLL26}
A. X. Do, N. Lam, G. Lu,
Sharp stability of the Heisenberg Uncertainty Principle: Second-order and curl-free field cases,
J. Funct. Anal. {\bf 290} (2026), no. 7, Paper No. 111321.

\bibitem{DEFFL}
J. Dolbeault, M. J. Esteban, A. Figalli, R. Frank, M. Loss,
Sharp stability for Sobolev and log-Sobolev inequalities, with optimal dimensional dependence,
Camb. J. Math. {\bf 13} (2025), no. 2, 359--430.

\bibitem{DLS07}
Y. Dong, G. Lu, L. Sun,
Global Poincar\'e inequalities on the Heisenberg group and applications,
Acta Math. Sin. (Engl. Ser.) {\bf 23} (2007), no. 4, 735--744.

\bibitem{DN25}
A. T. Duong, V. H. Nguyen,
On the stability estimate for the sharp second order uncertainty principle,
Calc. Var. Partial Differential Equations {\bf 64} (2025), no. 4, Paper No. 129, 24 pp.

\bibitem{Fat21}
M. Fathi,
A short proof of quantitative stability for the Heisenberg-Pauli-Weyl inequality,
Nonlinear Anal. {\bf 210} (2021), Paper No. 112403, 3 pp.

\bibitem{FIL16}
M. Fathi, E. Indrei, M. Ledoux,
Quantitative logarithmic Sobolev inequalities and stability estimates,
Discrete Contin. Dyn. Syst. {\bf 36} (2016), no. 12, 6835--6853.

\bibitem{FN19}
A. Figalli, R. Neumayer,
Gradient stability for the Sobolev inequality: the case $p \geq 2$,
J. Eur. Math. Soc. (JEMS) {\bf 21} (2019), no. 2, 319--354.

\bibitem{FZ-Duke}
 A. Figalli, Y. Zhang, Sharp gradient stability for the Sobolev inequality, Duke Math. J. {\bf 171} (2022), no. 12, 2407–2459.
 
\bibitem{FLW95a}
B. Franchi, G. Lu, R. L. Wheeden,
Representation formulas and weighted Poincar\'e inequalities for H\"ormander vector fields,
Ann. Inst. Fourier (Grenoble) {\bf 45} (1995), no. 2, 577--604.

\bibitem{FLW95b}
B. Franchi, G. Lu, R. L. Wheeden,
Weighted Poincar\'e inequalities for H\"ormander vector fields and local regularity for a class of degenerate elliptic equations,
Potential Anal. {\bf 4} (1995), no. 4, 361--375.

\bibitem{FLW96}
B. Franchi, G. Lu, R. L. Wheeden,
A relationship between Poincar\'e-type inequalities and representation formulas in spaces of homogeneous type,
Internat. Math. Res. Notices {\bf 1996}, no. 1, 1--14.

\bibitem{GW06}
P. Gir\~ao, T. Weth,
The shape of extremal functions for Poincar\'e-Sobolev-type inequalities in a ball,
J. Funct. Anal. {\bf 237} (2006), no. 1, 194--223.

\bibitem{GRO75}
L. Gross,
Logarithmic Sobolev inequalities,
Amer. J. Math. {\bf 97} (1975), 1061--1083.
\bibitem{Hen06}
A. Henrot,
Extremum Problems for Eigenvalues of Elliptic Operators,
Frontiers in Mathematics, Birkh\"auser Verlag, Basel, 2006.

\bibitem{Jerison-Duke}
D. Jerison,
The Poincar\'e inequality for vector fields satisfying H\"ormander's condition,
Duke Math. J. {\bf 53} (1986), no. 2, 503--523.

\bibitem{JL88}
D. Jerison, J. M. Lee,
Extremals for the Sobolev inequality on the Heisenberg group and the CR Yamabe problem,
J. Amer. Math. Soc. {\bf 1} (1988), no. 1, 1--13.

\bibitem{LLR25}
N. Lam, G. Lu, A. Russanov,
Stability of Gaussian Poincar\'e inequalities and Heisenberg Uncertainty Principle with monomial weights,
Math. Z. {\bf 312} (2026), no. 2, Paper No. 42.

\bibitem{LLR26}
N. Lam, G. Lu, A. Russanov,
Log-Sobolev and Beckner inequalities and stability of Poincar\'e inequality with weighted Gaussian measures,
arXiv:2604.16791.

\bibitem{Lu92}
G. Lu,
Weighted Poincar\'e and Sobolev inequalities for vector fields satisfying H\"ormander's condition and applications,
Rev. Mat. Iberoamericana {\bf 8} (1992), no. 3, 367--439.

\bibitem{Lu94}
G. Lu,
The sharp Poincar\'e inequality for free vector fields: an endpoint result,
Rev. Mat. Iberoamericana {\bf 10} (1994), no. 2, 453--466.

 

 

\bibitem{LW98}
G. Lu, R. L. Wheeden,
Poincar\'e inequalities, isoperimetric estimates and representation formulas on product spaces,
Indiana Univ. Math. J. {\bf 47} (1998), no. 1, 123--151.

\bibitem{MV21}
S. McCurdy, R. Venkatraman,
Quantitative stability for the Heisenberg-Pauli-Weyl inequality,
Nonlinear Anal. {\bf 202} (2021), Paper No. 112147, 13 pp.

\bibitem{VHN16}
V. H. Nguyen,
New approach to the affine P\'olya-Szeg\"{o} principle and the stability version of the affine Sobolev inequality,
Adv. Math. {\bf 302} (2016), 1080--1110.

\bibitem{Sz54}
G. Szeg\"{o},
Inequalities for certain eigenvalues of a membrane of given area,
J. Rational Mech. Anal. {\bf 3} (1954), 343--356.

\bibitem{Wat44}
G. N. Watson,
A Treatise on the Theory of Bessel Functions,
2nd ed., Cambridge University Press, Cambridge, 1944.

\bibitem{Wei56}
H. F. Weinberger,
An isoperimetric inequality for the $N$-dimensional free membrane problem,
J. Rational Mech. Anal. {\bf 5} (1956), 633--636.

\end{thebibliography}
\end{document}